\documentclass[10pt]{amsart}
\usepackage{amsfonts}
\usepackage{amssymb}
\usepackage{amsthm}
\usepackage{verbatim}
\usepackage[dvips]{graphicx}
\theoremstyle{plain}
\newtheorem{theorem}{Theorem}[section]
\newtheorem{lemma}[theorem]{Lemma}
\newtheorem{definition}[theorem]{Definition}
\newtheorem{corollary}[theorem]{Corollary}
\newtheorem{prop}[theorem]{Proposition}

\newcommand{\C}{\mathbb{C}}
\newcommand{\Z}{\mathbb{Z}}

\renewcommand{\P}{\mathbb{P}}
\newcommand{\bs}{\backslash}
\renewcommand{\epsilon}{\varepsilon}

\renewcommand{\phi}{\varphi}

\newcommand{\Q}{\mathcal{Q}}
\renewcommand{\H}{\mathcal{H}}

\newcommand{\M}{\mathcal{M}}
\newcommand{\tM}{\tilde{M}}
\newcommand{\tq}{\tilde{q}}

\newcommand{\T}{\mathcal{T}}

\renewcommand{\l}{\lambda}
\renewcommand{\a}{\alpha}
\renewcommand{\b}{\beta}
\newcommand{\g}{\gamma}

\begin{document}
\title[Connected components]{Connected components of strata of quadratic differentials over Teichmuller space.}
\author {Katharine C. Walker}
\email{kaceyw@umich.edu} 
\begin{abstract} In this paper, we study connected components of strata of the space of quadratic differentials lying over $\T_g$. We use certain general properties of sections of line bundles to put a upper bound on the number of connected components, and a generalized version of the Gauss map as an invariant to put a lower bound on the number of such components. For strata with sufficiently many zeroes of the same order we can state precisely the number of components.   
\end{abstract}

\maketitle
\section{Introduction}
We define $\Q_g$ to be the space of all pairs $(M, q)$ where $M$ is an element of $\T_g$ and $q$ is a quadratic differential on $M$. Alternatively let $K=K_M$ be the holomorphic cotangent bundle on $M$, let $H^0(M, K^2)$ be the space of sections of $K_M^2$, and define $\Q_g$ to be the total space of the bundle over $\T_g$ with fiber $H^0(M, K^2)$, where by custom we remove sections that are the squares of sections of $K$. We may define an analogous space over the moduli space of curves of genus $g$, $\M_g$, which we will refer to as $\Q D_g$. 
 
Both $\Q_g$ and $\Q D_g$ carry a natural stratification given by the orders of the zeroes of the quadratic differential. Let $\l=(k_1, k_2,...,k_n)$ be a partition of $4g-4$ and define $\Q_g(k_1, k_2,...,k_n)=\Q_\l$ to be the subspace of $\Q_g$ of $(M,q)$ such that $q$ has $n$ zeroes, with multiplicities $k_1,...,k_n$.  We use the notation $(k_1^{m_1},k_2^{m_2},...,k_n^{m_n})$ to abbreviate strata with large numbers of zeroes of the same order. 

Interestingly, while most of the $\Q D_\l$ are connected, there are three families of strata of high codimension that have two components. In each case, one of the components consists solely of \textit{hyperelliptic} quadratic differentials - $(M,q)$ such that $M$ is hyperelliptic and $q$ is invariant under the hyperelliptic involution. Lanneau proves this in \cite{L1}. Similarly, in \cite{KZ} Kontsevich and Zorich classify the connected components of strata of Abelian differentials over $\M_g$. Again most strata are connected but certain strata of high codimension may have up to three components. These are classified again by hyperellipticity, but also by even and odd spin structure. In \cite{L2}, Lanneau shows that spin structure cannot be used to classify connected components of the $\Q D_\l$. 

In this paper we are interested in the connected components of the $\Q_\l$. In Section \ref{sec:up} we use certain general facts about sections of line bundles to show that for $\l$ with sufficiently many zeroes of the same order we may put an upper bound on the number of connected components. In Section \ref{sec:low} we use an analog of spin structure to put a lower bound on the number of connected components of $\Q_\l$ for $\l=(k_1,...,k_n)$ with all $k_i$ even. The main result is as follows. 
\begin{theorem} Let $g \ge 2$ and $m \ge g$. Then: 
\begin{enumerate}
\item  Any stratum of the form $\Q_g(1^m, k_1^{n_1},...,k_l^{n_l})$ is connected.
\item For $k_1,...,k_l$ all even, any stratum of the form $\Q_g(2^m, k_1^{n_1},...,k_l^{n_l})$ has exactly $2^{2g}-1$ connected components. 
\end{enumerate}
\end{theorem}

Each of the components of this second family of strata is made up of sections that are squares of sections of one of the $2^{2g}$ line bundles that square to $K^2$, excluding the squares of Abelian differentials.

\section{Background} \label{sec:back}
Given a quadratic differential, $(M,q)$, one may construct a double cover of $M$, possibly ramified, $\pi: \tM \to M$, and define a quadratic differential $\tq:=\pi^*(q)$ on $\tM$ by taking the pullback of $q$. If $q$ has a singularity of order $k$ at a ramification point of $\pi$, then $\tq$ has a singularity of order $2k+2$ ($k \in \{-1,0,1,...\}$, where 0 corresponds to a regular point of $q$) at that point. One may see this by recalling that a zero of order $k$ corresponds to a cone angle of $(k+2)\pi$; thus, ramifying at that point gives a cone angle of $(2k+4) \pi$ or a singularity of order $2k+2$. In certain cases $\pi_*(q)$ may then be the square of an Abelian differential. We define strata of Abelian differentials analogously to strata of quadratic differentials, and denote them by $\H_g(l_1,l_2,...,l_n)$, $\sum l_i=2g-2$. 

\begin{definition} Let $(M,q) \in \Q_\l$ be a quadratic differential that is not the square of an Abelian differential. The \textbf{canonical double cover} $\pi: \tM \to M$ is the minimal cover such that $\pi^*(q)$ is the square of an Abelian differential, $\omega$, on $\tM$. The ramification points of $\pi$ are the singularities of $q$ of odd order (of which there may be none).
\end{definition}

We will use this construction to associate a unique Abelian differential to any quadratic differential. Notice that `nearby' quadratic differentials will map to the same stratum of Abelian differentials so, choosing a branch of the square root map, we get a local mapping $\Q_g(k_1,..,k_n) \to \H_{\tilde{g}}(\tilde{k}_1, \tilde{k}_2,...,\tilde{k}_m)$ (see \cite{L1}, for example, for a proof). 

A \textit{spin structure} on $M$ is a choice of a line bundle that squares to the canonical bundle, $K_M$ - $L \in Pic(M)$ such that $L^{\otimes 2} = K_M$. On a curve of genus $g$ there are $2^{2g}$ different spin structures. The \textit{parity} of the spin structure is the dimension of its space of holomorphic sections, $H^0(M,L)$, mod 2. In \cite{KZ} Kontsevich and Zorich use parity to differentiate connected components of Abelian differentials. 

We can also define the spin structure in a more topological way. If we let $P \to M$ be the unit tangent bundle on $M$, a spin structure on $M$ may be viewed as a double cover $Q \to P$ whose restriction to each fiber of $P$ is the standard double cover $S^1 \to S^1$. The double cover gives a monodromy map $\pi_1(P) \to \Z_2$, which because $\Z_2$ is Abelian induces a map $\xi:H_1(P, \Z_2) \to \Z_2$. The map has non-zero value on the cycle representing the fiber of $P$ by definition. 

For a class in $H_1(M, \Z_2)$ that contains a simple closed smooth curve, there is a natural lift to a class in  $H_1(P, \Z_2)$, given as follows. Let $\g \in [\g]$ be a simple closed curve, and take each point in $\g$ to its tangent vector. Call the resulting curve (the \textit{framed curve} of $\g$) $\hat{\g}$. The map $[\g] \mapsto [\hat{\g}]$ is known to be well-defined, and thus $\xi$ associates a value in $\Z_2$ to each homology class of simple closed curves. 
 
Finally, recall that an Abelian differential, $\omega$ on $M$, corresponds to a flat structure on $M$ with isolated cone angles of degree $2k\pi$ and trivial holonomy. Away from the finite set of points with cone angles, $M$ has a well-defined notion of horizontal and thus there is a well-defined Gauss map, taking unit tangent vectors to points on the unit circle: 
$$G: T(M \bs \omega^{-1}(0)) \to S^1$$
A smooth simple closed curve has a well-defined framed curve, and thus has a degree under $G$. This degree coincides with the value of the curve under $\xi$, (mod 2).  
\begin{figure} 
\begin{center}
\includegraphics[width=.4 \textwidth]{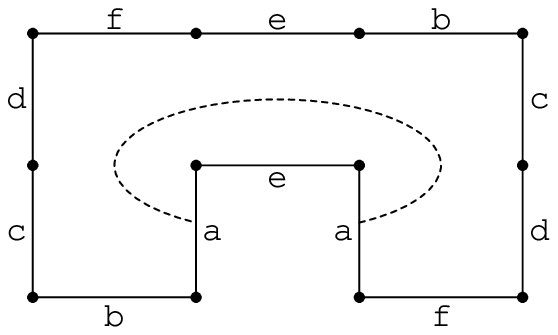} 
\ \ \ \ \ 
\includegraphics[width=.4 \textwidth]{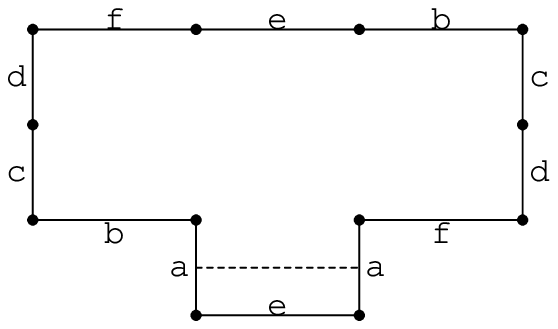}
\end{center}
\caption{2 elements of $\H_3(4)$ with oriented line fields induced by $dz$ and curves whose degrees under the Gauss map are $1$ and $0$ respectively.}
 \label{fig:gauss}
\end{figure}

For example, consider Figure \ref{fig:gauss}. This shows two genus 3 surfaces which have oriented line fields induced by $dz$. Thus we have two Abelian differentials, elements of $\H_3(4)$. The degree of the curve is $\pm 1$ (depending on choice of orientation) in the left example, and 0 in the right example.  

While we cannot use spin structure to classify connected components of the $\Q_\l$, notice that there are also $2^{2g}$ choices of square roots of $K^2$, and that if $\l$ has zeroes of only even order, any $(M,q) \in \Q_\l$ will be the square of a section in one of these square root bundles. In Section \ref{sec:low}, we use the degree of a basis of curves in $H_1(M, \Z)$ under an analog of the Gauss map to classify connected components of such $\Q_\l$.  

Finally, the following theorem from \cite{L1} classifies the connected components of all $\Q D_\l$.

\begin{theorem} [Lanneau]  \label{thm:lanneau}
For $g \ge 3$ the following strata have two connected components: 
\begin{enumerate}
\item $\Q D_g(4(g-k)-6, 4k+2)$, $k \ge 0, g-k \ge 2$
\item $\Q D_g((2(g-k)-3)^2, 4k+2)$, $k \ge 0, g-k \ge 1$
\item $\Q D_g((2(g-k)-3)^2, 2k+1^2)$, $k \ge 0, g-k \ge 2$
\end{enumerate}
and the rest have one component. 
For $g=0,1$, all strata are connected, and for $g=2$, $\Q D_2(3,3,-1,-1)$ and $\Q D_2(6,-1,-1)$ have two components, but all others are connected.
\end{theorem}

For any $\l$,  $\Q_\l \to \Q D_\l$ is a covering map (with fiber $\Gamma_g$, the standard mapping class group of genus $g$). This theorem then tells us that in most cases, $\Q_\l$ must lie over a connected space.  

\section{Upper Bounds on Connected Components} \label{sec:up}
In this section we show that strata of the form $\Q_g(1^m, k_1^{n_1},...,k_l^{n_l})$, $m \ge g$, are connected, and strata of the form $\Q_g(2^{m},k_1^{n_1},...,k_l^{n_l})$, $m \ge g$ and $k_1,....,k_l$ all even, have at most $2^{2g}-1$ connected components. These correspond to the $2^{2g}-1$ line bundles, excluding $K$, that square to $K^2$.  

Following Kontsevich and Zorich in \cite{KZ}, for any finite sequence of positive integers $\l=(k_1,...,k_n)$ (where $k_i$ may equal $k_j$ and $\sum k_i=4g-4$), we define $\Q_g^{num}(k_1,...,k_n)$ or $\Q_\l^{num}$ to be the space of quadratic differentials with \textit{numbered} zeroes, such that the $i$th zero is of order $k_i$. $\Q_g^{num}(k_1,...,k_n)$ is a manifold and a finite cover of  $\Q_g(k_1,...,k_n)$ (if for all $i,j$ $k_i \ne k_j$, then $\Q_\l^{num}=\Q_\l$), and the advantage of this space is that we may distinguish between different zeroes of the same order. Since it is a covering map, any connected component of $\Q_g^{num}(k_1,...,k_n)$ will remain connected under projection to $\Q_g(k_1,...,k_n)$. Also, for any $\l$ define $\P \Q_\l^{num}$ to be the projectivization of $\Q_\l^{num}$ given by identifying all constant multiples of a particular quadratic differential. $\P \Q_\l^{num}$ will have the same number of components as $\Q_\l^{num}$; consequently, an upper bound on the number of connected components of $\P \Q_\l^{num}$ is also an upper bound on the number of connected components of $\Q_\l$.

Let $\l=(1^m,k_1^{n_1},...,k_l^{n_l})$ and note that no quadratic differential with zeroes of the pattern given by $\l$ can be the square of an Abelian differential, so $\Q_\l$ is the union over all $M \in \T_g$ of the subsets of $H^0(M,K^2)$ of sections with zeroes of the correct orders. We assume that $m=(4g-4-(\sum_{i=1}^l k_i n_i)) \ge g$. Define $Sym^n(M)$ to be the space of n-tuples of distinct unordered points on $m$ ($M^n/S_n$ minus the fat diagonal), and let $Sym^n$ the bundle over $\T_g$ with fiber $Sym^n(M)$ over each $M$. Similarly let $Pic^{n}(M)$ denote the space of line bundles of degree $n$ on $M$, and recall that this will be a $g$ dimensional Abelian variety for any $n$. Again let $Pic^n$ be the bundle over $\T_g$  with fiber $Pic^n(M)$ over each $M$, 
which is of dimension $g$ for any $n$. From classical algebraic geometry (see \cite{GH}, for example) we have the Abel-Jacobi map 
$$AJ:Sym^n(M) \to Pic^n(M)$$
which is given by thinking of a set of $n$ points on $M$ as a divisor, and taking the divisor to its associated line bundle (i.e. the line bundle for which the points are the zeroes of a holomorphic section). We may extend this map to a map $Sym^n \to Pic^n$, which we will again call the Abel-Jacobi map. 

We also have the following map:
$$\phi:\P \Q^{num}_\l \to Sym^{n_1} \times... \times Sym^{n_l} \times Sym^{m-g} \times Sym^g$$
where the products are fiberwise (in other words, we define $Sym^{n_1} \times Sym^{n_2}$ to be the space with $Sym^{n_1}(M) \times Sym^{n_2}(M)$ over every $M \in \T_g$). The map takes $(M,q)$ to the zeroes of $q$, where the zeroes of order $k_i$ are taken to $Sym^{n_i}$ and the zeroes of order 1 are split into two parts, with the first $m-g$ mapped to $Sym^{m-g}$. This is well-defined because we are working in $ \P \Q^{num}_\l$. We will be especially interested in the restriction: 
$$\phi':\P \Q^{num}_\l \to Sym^{n_1} \times... \times Sym^{n_l} \times Sym^{m-g} $$
This is because $\P \Q^{num}_\l$ has dimension $2g-3+\sum_1^ln_i+m$, as does $Sym^{n_1} \times... \times Sym^{n_l} \times Sym^{m-g}$. For the rest of the section we will let $S=Sym^{n_1} \times... \times Sym^{n_l} \times Sym^{m-g}$. 

\begin{theorem} \label{thm:c1}
Let $g \ge 2$. For $\l=(1^m, k_1^{n_1},...,k_l^{n_l})$ and $m \ge g$, $\Q_\l$ is connected. 
\end{theorem}
By the discussion above it suffices to prove that $\P \Q^{num}_\l$ is connected. To prove this we will need the following lemma. 

\begin{lemma} \label{lem:c1}
There exists $U \subset \P \Q_\l^{num}$ such that $\P\Q_\l^{num} \bs U$ is of positive codimension, $\phi'$ is injective on $U$, and $\phi'(U)$ is connected. 
\end{lemma}

\begin{proof}
We construct $U \subset \P \Q^{num}_\l$ as follows. Pick an element of $\P \Q^{num}_\l$ and divide its zeroes into two divisors: the $n_1+...+n_l$ zeroes of higher order plus the first $m-g$ zeroes of order 1, and the last $g$ zeroes of order 1. These two divisors give us holomorphic sections, $s$ of some line bundle $L_s \in Pic^{3g-4}$ and $s'$ of $K^2/L_s \in Pic^g$ (we require $g \ge 2$ so $3g-4 > 0$). Let $U$ be the subset of $\P \Q^{num}_\l$ such that $h^0(M, K^2/L_s)=1$. This is a generic condition (Riemann-Roch implies that for a line bundle of degree $d$ the dimension of its space of holomorphic sections will be $d-g+1$) so $\P \Q^{num}_\l \bs U$ is of positive codimension. 

$\phi'$ is injective on $U$ because if two elements of $\P \Q^{num}_\l$ get mapped to the same element of $S$, they get mapped to the same $L \in Pic^{3g-4}$ by a restriction of the Abel-Jacobi map to $S$. Then they also get mapped to the same element of $Pic^g$, $K^2/L$, by a restriction of the Abel-Jacobi map to $Sym^g$. By assumption $K^2/L$ has only one holomorphic section up to scaling; thus our original two elements of $\P \Q^{num}_\l$ must have been the same. 

To show $\phi'(U)$ is connected it suffices to show that $S \bs \phi'(U)$ is of positive complex codimension in $S$ (because $S$ is a connected complex manifold). However, any element of  $S$ gives us a divisor $D$, which (with zeroes weighted appropriately) gives us a section, $s_D$, of a line bundle $L_D \in Pic^{3g-4}$. $K^2/L_D$ will generically have a section, $s_{D'}$ with $g$ distinct zeroes, $D'$, that are also distinct from those of $D$, and also generically $h^0(M, K^2/L_D)=1$. When this is the case $s_{D} \otimes s_{D'}$ is an element of $U$ and $D$ is in the image of $\phi'$.
\end{proof}

\begin{proof}[Proof of Theorem]
Recall that a holomorphic injective map has a holomorphic inverse. Thus the fact that $\phi'$ is injective on $U$ and $\phi'(U)$ is connected implies that $\phi'^{-1}(\phi'(U))=U$ is also connected. This implies that there can be only one full-dimensional connected component of $\P\Q_\l^{num}$.

However, all the components of $\P\Q_\l^{num}$ must be of the same dimension, because by Theorem \ref{thm:lanneau} $\Q D_\l$ has only one component for any of $\l$ as in the statement of the theorem, and all components of $\Q_\l$ must cover it. Thus there can be only one connected component of $\P\Q_\l^{num}$.
\end{proof}

Now let $\l=(2^m,k_1^{n_1},..., k_l^{n_l})$, with $m \ge g$ and $k_1,...,k_l$ even, and let $n=n_1+...+n_l$. A similar argument to the one above will allow us to put an upper bound on the number of connected components of $\Q_l$. We again have the following maps, defined as above:
$$\phi:\P\Q^{num}_\l \to Sym^{n_1} \times... \times Sym^{n_l} \times Sym^{m-g} \times Sym^g$$
$$\phi':\P\Q^{num}_\l \to Sym^{n_1} \times... \times Sym^{n_l} \times Sym^{m-g} $$

We cannot use precisely the same argument as above because the double zeroes do not form a subset of full-dimension of the sections of some line bundle. However, for a particular $M$ recall that $K_M^2$ has $2^{2g}$ square roots in $Pic^{2g-2}(M)$. Call them $K_1, K_2,...,K_{2^{2g}}=K_M$. Since $Pic^{2g-2}$ is a trivial bundle over $\T_g$ a choice of labeling of these bundles over a single $M$ yields a consistent choice of labeling over all other surfaces; thus, when we refer to $K_i$ we will mean a section of $Pic^{2g-2}$. Let $(\P \Q_\l)_{K_i} = \{ q | \sqrt{q} \textrm{ is a section of } K_i \}$. The set $\{ q | \sqrt{q} \textrm{ is a section of } K \}$ consists of squares of Abelian differentials, and it is customary to consider them separately, so we have $2^{2g} - 1$ subsets $(\P \Q_\l)_{K_i}$, all of full dimension in $\P\Q_\l$.

\begin{theorem} \label{thm:c2}
$\Q_\l$ has at most $2^{2g}-1$ connected components.
\end{theorem}

To prove this we show that each of the $(\P\Q_\l)_{K_i}$ is connected, and use the following lemma. 

\begin{lemma}
There exists $U_i \subset (\P \Q_\l^{num})_{K_i}$ such that $(\P\Q_\l^{num})_{K_i} \bs U_i$ is of positive codimension, $\phi'$ is injective on $U_i$, and $\phi'(U_i)$ is connected. 
\end{lemma}

\begin{proof}
The proof is very similar to that of Lemma \ref{lem:c1}. We construct $U_i$ as follows. Pick an element of $(\P \Q^{num}_\l)_{K_i}$, which we regard as projective section of $K_i$ with the extra data of the ordering of the zeroes (now of order $(1^m, (k_1/2)^{n_1},...,(k_l/2)^{n_l})$). Divide its zeroes into two divisors: the $n$ zeroes of higher order plus the first $m-g$ zeroes of order 1, and the last $g$ zeroes of order 1. These two divisors give us holomorphic sections, $s$ of some line bundle $L_s \in Pic^{g-2}$ and $s'$ of $K_i/L_s \in Pic^g$. Let $U_i$ be the subset of $(\P \Q^{num}_\l)_{K_i}$ such that $h^0(M, K_i/L_s)=1$. This is a generic condition so $(\P \Q^{num}_\l)_{K_i} \bs U_i$ is of positive codimension. 

$\phi'$ is injective on $U_i$ because if two elements of $(\P \Q^{num}_\l)_{K_i}$ get mapped to the same element of $S$, they get mapped to the same $L$ by the Abel-Jacobi map. Then they also get mapped to the same element of $Pic^g$, $K_i/L$, by the restriction of the Abel-Jacobi map to $Pic^g$. By assumption $K_i/L$ has only one holomorphic section up to scaling; thus our original two elements of $\P \Q^{num}_\l$ must have been the same. 

To show $\phi'(U_i)$ is connected it suffices to show that $S \bs \phi'(U_i)$ is of positive codimension in $S$. However, any element of $S$ gives us a divisor $D$, which gives us the section, $s$, of a line bundle $L_s \in Pic^{g-2}$. $K_i/L_s$ will generically have a section, $s'$, with $g$ distinct zeroes that are distinct from those of $D$, and also generically $h^0(M, K_i/L)=1$. When this is the case $s \otimes s'$ is an element of $U_i$ and $D$ is in the image of $\phi'$.
\end{proof}

\begin{proof}[Proof of Theorem]
As in the proof of Theorem \ref{thm:c1} the fact that $\phi'$ is injective on $U_i$ and $\phi'(U_i)$ is connected implies that $\phi'^{-1}(\phi'(U_i))=U_i$ is also connected. This implies that there can be at most $2^{2g}-1$ full-dimensional connected components of $\P\Q_\l^{num}$.

However, all the components of $\P\Q_\l^{num}$ must be of the same dimension, because by Theorem \ref{thm:lanneau} $\Q D_\l$ has only one component for any of $\l$ as in the statement of the theorem, and all $\Q_\l$ must cover it. Thus there can be only $2^{2g}-1$ connected components of $\P\Q_\l^{num}$.
\end{proof}

An analog of Theorem \ref{thm:c2} may be used to show that $\Q_\l$ has at most $3^{2g}$ components for $\l=(3^m, k_1^{n_1},...,k_l^{n_l})$, $m \ge g$ and $k_i$ divisible by 3. For integers larger than 3 this technique no longer applies, as $k^m$ will be greater than $4g-4$ for $m \ge g$.

Finally, we note that one would expect the upper bound of Theorem \ref{thm:c2} to be satisifed. This is because the squares of sections of a particular $K_i$ in any connected component will be a closed set in that component (the limit of a sequence of sections of a particular bundle will still be a section of that bundle), and a connected set cannot be the union of multiple disjoint closed subsets. In the next section we construct a geometric invariant of a connected component, which gives another proof of this fact. 


\section{Lower Bounds} \label{sec:low}
In this section we first define an invariant of the homology class of a curve on a surface, and show that for $\l=(k_1,..,k_n)$ with all $k_i$ even, it defines an invariant of a connected component of $\Q_\l$. We then use the invariant to show that for these $\l$ $\Q_\l$ has multiple connected components. 

Let $\l=(k_1,...,k_n)$ with all $k_i$ even. For $(M,q) \in \Q_\l$ let $\pi: \tM \to M$ be the canonical double cover, and let $\omega^2=\pi^*(q)$, with $\omega$ an Abelian differential on $\tM$. Let $P=\{p_1,...,p_n\}$ denote the set of zeroes of $q$. Since $q$ has no odd zeroes the cover is not ramified. Thus $\pi: \tM \to M$ is a  two-sheeted covering space and we have a monodromy representation: $$Ga:\pi_1(M) \to \Z_2$$
Since $\Z_2$ is Abelian, the representation factors through $\pi_1(M) \to H_1(M, \Z) \to \Z_2$, and we will refer to the second map as $Ga$ as well. 
If a cycle, $[\g] \in H_1(M, \Z)$, is represented by a smooth, simple, connected, closed curve, $ \g \subset M \bs P$, then we may calculate the monodromy geometrically as follows. Consider the flat metric on $M \bs P$ induced by $q$, and the corresponding connection on $TM$. $[\g]$ is in the kernel of of $Ga$ if and only if parallel transport along $\g$  brings a vector, $v$, back to itself. (The other option is that $v$ is brought back to $-v$.) We refer to the map as $Ga$ because it corresponds to the degree of a generalized version of the Gauss map on $M$, (mod 2).  Notice that one may not not use this as an invariant of homology classes on quadratic differentials with odd zeroes because it is possible to create two smooth simple curves in the same class $H_1(M, \Z)$ that are in different classes in $H_1(M \bs P, \Z)$, differ by a small loop around an odd zero, and thus have different monodromies. Also notice that $Ga$ is a group homomorphism, and $Ga([\gamma + \eta]) = Ga([\g])+ Ga([\eta])$.

\begin{lemma} \label{lem:odd}
Let $(M,q)$ be a quadratic differential with all even zeroes. Any basis for $H_1(M, \Z)$ contains at least one cycle with non-zero monodromy. 
\end{lemma}
\begin{proof}
If all cycles in a basis for $H_1(M,\Z)$ have value $0$ under $Ga$, then all non-trivial cycles must also have value 0, as they can be written as a sum of the basis cycles. By assumption $q$ is not the square of an Abelian differential on $M$. Thus, there exists some $[\g]$ such that $Ga([\g])=1$, and $[\g]$ is not the trivial class because $H_1(M, \Z) \to \Z_2$ is a group homomorphism.
\end{proof}

We now wish to use $Ga$ to create an invariant of connected components of $\Q_\l$. To do so, recall that if we choose a cycle on a particular surface $M_0 \in \T_g$, it gives a well-defined cycle on any other $M \in \T_g$. One way to define this cycle is to construct the bundle $H \to \T_g$ with fiber $H_1(M, \Z)$ over each $M \in \T_g$, which as a discrete bundle has a well-defined notion of parallel transport. Since $\T_g$ is simply connected parallel transport does not depend on choice of path. 
\begin{prop} \label{prp:pt}
Let $(M_0, q_0), (M_1,q_1)$ be in the same connected component, $\Q_\l^0$, of
$\Q_\l$. Pick $[\g_0] \in H_1(M_0, \Z)$, and let $[\g_1] \in H_1(M_1, \Z)$ be the corresponding cycle on $M_1$.  Then $Ga([\g_1])=Ga([\g_0])$ 
\end{prop}

\begin{proof} 
We have the projection $\pi:\Q_\l \to \T_g$, and may thus consider the bundle $\pi^*H$, which is again a discrete bundle on which parallel transport is well-defined. Pick a path $p:[0,1] \to \Q_\l^0$ from $(M_0, q_0)$ to $(M_1,q_1)$, $p(t)=(M_t,q_t) \in \Q_\l^0$, $0 \le t \le 1$.   Let $[\g_t]$ the class of $H_1(M_t, \Z)$ obtained from $[\g_0]$ via parallel transport in $\pi^*(H)$. The range of $Ga$ is $\Z_2$, so we have a continuous map, $[\g_t] \mapsto Ga([\g_t])$, into a discrete space, which must be constant. Since $[\g_0]$ will be taken to $[\g_1]$ along $p$, $Ga([\g_1])=Ga([\g_0])$. 
\end{proof}

\begin{corollary} \label{cor:comp}
Pick a basis, $B=(\alpha_1,...,\alpha_g, \beta_1,..,\beta_g)$, for $H_1(M_0, \Z)$ and let $B'$ be the corresponding basis for $H_1(M_1, \Z)$. $(M_1,q_1)$ cannot be in the same connected component of $\Q_\l$ as $(M_0,q_0)$ if $Ga(B') \ne Ga(B)$.  
\end{corollary}

\begin{proof}
Suppose they are in the same component. Pick a path from $(M_0,q_0)$ to $(M_1,q_1)$, lying in $\Q_\l^0$, and parallel transport $B$ along the path to $B'$. $B'$ will be a basis for $H_1(M_1, \Z)$ and by Proposition \ref{prp:pt}  Ga($B$) = Ga($B'$). 
\end{proof}

We would like to construct bases for $H_1(M, \Z)$ with different values under $Ga$ and apply Corollary \ref{cor:comp} to show that there exist multiple components of $\Q_\l$. To do so, we will consider the action of the standard mapping class group, $\Gamma_g$, on $\Q_\l$, and in particular the action of Dehn twists. 

\begin{lemma} \label{lem:twist}
Let $(M,q) \in \Q_\l$, and let $\gamma$ be a simple closed curve on $M$. If $\gamma \cdot \delta = 1$, Dehn twisting around $\gamma$ takes $\delta$ to a new curve $\delta'$ such that $Ga([\delta']) = Ga([\delta])+ Ga([\gamma])$. If $\gamma \cdot \delta = 0$ then Dehn twisting around $\gamma$ has no effect on $Ga([\delta])$.
\end{lemma}

\begin{proof}
If $\gamma \cdot \delta=1$ then $\delta'$ is homologous to $\delta + \gamma$. If $\gamma \cdot \delta=0$, then Dehn twisting around $\gamma$ just takes $\delta$ to $\delta$. 
\end{proof}

As an example, we consider the structure and interactions of the following strata: $\Q_1(-2,1,1)$, $\H_1(\emptyset)$, and $\Q_1(-2,2)$. (In general we restrict ourselves to strata of holomorphic quadratic differentials, but genus one allows for explicit calculation and $\Q_1(-2,2)$ is the simplest stratum with poles of all even order.) We think of a genus one surface as $\C/\Lambda$, where $\Lambda$ is the lattice generated by $\{ 1,\tau \}$ and $\tau$ is an element of the upper half plane. $\T_1$ may therefore be identified with the upper half plane. We will denote by $\a$ the homology class of the curve $[0,1] \to \C/\Lambda$, $t \mapsto t\tau$, and by $\b$ the homology class of the curve $[0,1] \to \C/\Lambda$, $t \mapsto t$.        
On a genus one surface $K^2$ is just the trivial bundle, and the sum of the zeroes of any quadratic differential must add to zero (mod $\Lambda$). Notice that a function on $\C/\Lambda$ cannot have a single zero and a single pole because they would be forced to be in the same place, so $\H_1(1, -1)$ is empty and we need not worry about elements of $\Q_1(-2,2)$ being squares of Abelian differentials. On the other hand, any quadratic differential with no zeroes or poles is of the form $k dz^2$ and must have square root $\sqrt{k} dz \in \H_1(\emptyset)$, so $\Q_1(\emptyset)$ is also empty. 

\begin{figure}\label{fig:dbl}
\begin{center}
\includegraphics[width=.35 \textwidth]{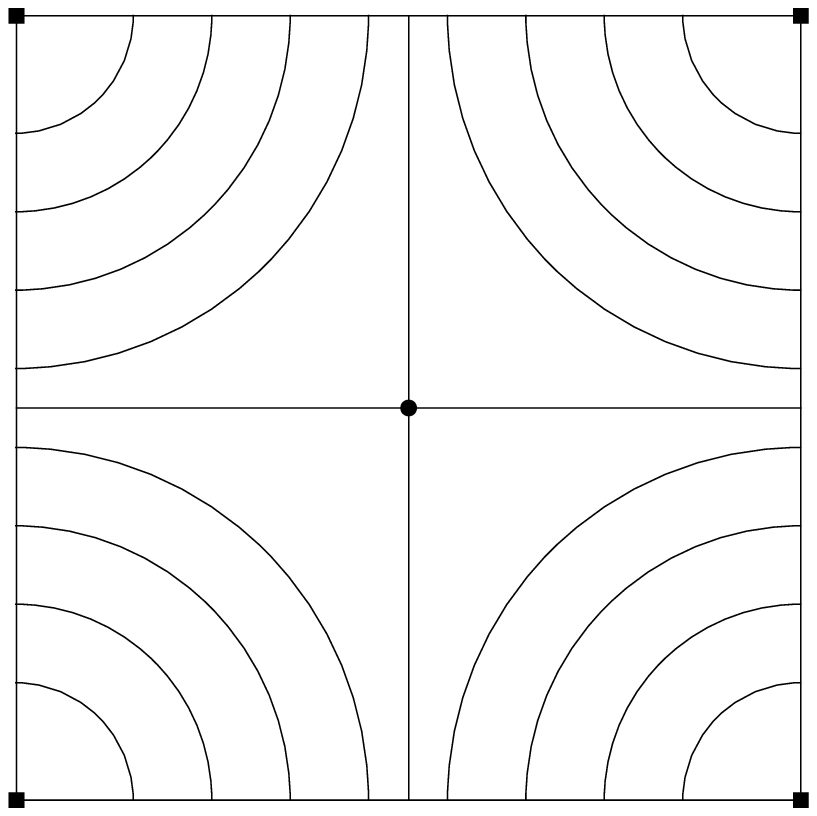}
\ \
\includegraphics[width=.55 \textwidth]{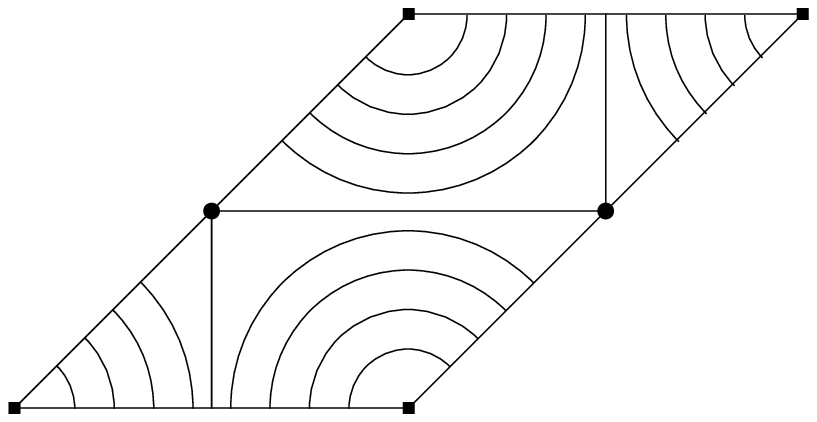}
\end{center}
\caption{The horizontal foliations of quadratic differentials given by the Weierstrass $\mathfrak{p}$ function for $\tau=i$ and $\tau=1+i$. Square points represent double poles, and circular points represent double zeroes.} 
\label{fig:uhp}
\end{figure}

Because of the triviality of the canonical bundle, any triplet of points, $(x_1,x_2,x_3)$ such that $x_1+x_2=x_3$ (mod $\Lambda$), will be the zeroes and double pole of a quadratic differential on $M = \C / \Lambda$. However, we have to factor out the 1 dimensional automorphism group of any torus, so we assume that the double pole of $q$, $x_3$, is at 0. This forces the two remaining zeroes of $q$ to be at $x$ and $-x$. In general $x$ and $-x$ are distinct and $(M,q) \in \Q(-2,1,1)$; however, if $x=1/2, \tau/2 $, or $\frac{1+\tau}{2}$ then $x=-x$ and $(M, q) \in \Q_1(-2,2)$. If $x=0$ then the section has no zeroes or poles and is in $\H_1(\emptyset)$. 
Then let $\mathcal{C}_1$ be the \textit{universal curve} over $\T_1$ - in other words, the space constructed by putting over each point in Teichmuller space the corresponding curve. The `zero section' of $\mathcal{C}_1$ is $\P\H_1(\emptyset)$, which may then be identified with $\T_1$, or the upper half plane. $\P \Q_1(-2,2)$ is the union of the sections of $\mathcal{C}_1$ corresponding to $1/2, \tau/2$, and $\frac{1+\tau}{2}$, which thus has 3 components, each a copy of $\T_1$, and containing surfaces whose homology bases have image $(0,1), (1,0)$, and $(1,1)$ (not necessarily respectively) under $Ga$. $\P \Q_1(-2,1,1)$ is $\mathcal{C}_1$ minus the above 4 sections, with points of the form $(M,x)$ and $(M,-x)$ identified. It is therefore connected.    

Figure \ref{fig:dbl} shows the singularities and horizontal foliations of two quadratic differentials, on $\tau=i$ and $\tau=1+i$, where the second is the result of Dehn twisting the first along $\b$. The quadratic differentials on the two surfaces are given by the Weierstrass $\mathfrak{p}$ function, which is defined on the curve with modulus $\tau$ as follows:
$$\mathfrak{p}(z)=\frac{1}{z^2}+\sum_{m,n \in \Z} \left( \frac{1}{(z-m-n\tau)^2} - \frac{1}{(m+n\tau)^2} \right)$$
 where $ (m,n) \ne (0,0)$. This is a function on the plane but is periodic with respect to $\Lambda$, and thus descends to a function on $\C / \Lambda$. We may then define a quadratic differential on $\C / \Lambda$, $q_{\tau}=\mathfrak{p}(z)dz^2$. For most $\tau$ $q_{\tau}$ has two single zeroes and a double pole at $0$; however, for $\tau = i$ it has a single double zero at $1/2+i/2$, and for $\tau=1+i$ it again has a double zero at at the same location. We can see that for $\tau=i$, $Ga(\alpha)=Ga(\b)=1$, and for $\tau=1+i$, $Ga(\a)=0$ and $Ga(\b)=1$, as predicted by Lemma \ref{lem:twist}. 

We wish to apply a similar analysis to the above to a wider variety of examples. For any $(M,q) \in \Q_\l$, Lemma \ref{lem:odd} implies that any basis for $H_1(M, \Z)$ will contain  a cycle, $[\g]$, such that $Ga([\g])=1$. We may then create $(M',q') \in \Q_\l$ by Dehn twisting around $\g$ (essentially, picking a new homology basis for $M$), and Lemma \ref{lem:twist} and Corollary \ref{cor:comp} imply that $(M,q)$ and $(M',q')$ must be in different connected components of $\Q_\l$. In the next theorem we show that we can use Dehn twists to create $(M,q)$ with homology basis that has any image under $Ga$ except $(0,0,0,...,0)$. Thus $\Q_\l$ will have a minimum of $2^{2g}-1$ components. 

\begin{theorem} \label{thm:mcg}
Let $\Q_\l=\Q_g(k_1,...,k_n)$ with the $k_i$ all even. Then $\Q_\l$ has at least $2^{2g}-1$ connected components. 
\end{theorem}

 \begin{figure}
\begin{center}
\includegraphics[width=.9 \textwidth]{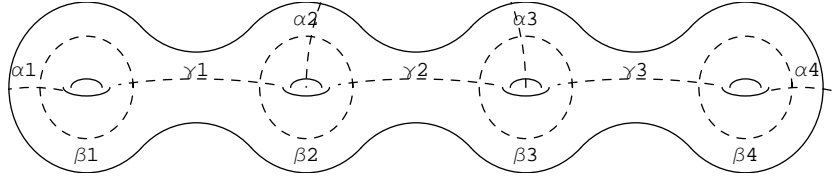}
\end{center}
\caption{A genus 4 surface with curves labeled as in Theorem \ref{thm:mcg}.}
\label{fig:genusg}
\end{figure}

\begin{proof}

Let $(M,q) \in \Q_\l$. By Lemma \ref{lem:odd} $(M,q)$ has a symplectic basis for $H_1(M, \Z)$ with at least one cycle of odd degree. Call the basis $(\a_1,...,\a_g, \b_1,...,\b_g)$ and let $Ga(\a_1,...,\a_g, \b_1,...,\b_g)=(a_1,...,a_g, b_1,...,b_g)$. We also construct curves $\gamma_1,...,\gamma_{g-1}$ on $M$ as in Figure \ref{fig:genusg} and assign them degrees $c_1,...,c_{g-1}$. Note that $\gamma_i$ is homologous to $\a_i+\a_{i+1}$ (with an appropriate choice of orientation of the curves) so $Ga(\gamma_i)=Ga(\alpha_i)+Ga(\alpha_{i+1})$ (mod 2). To prove the theorem we show that one can repeatedly apply Dehn twists to get a homology basis of any of the $2^{2g}$ possible parities, except $(0,0,...,0)$, and then apply Corollary \ref{cor:comp}. To construct such a set of bases it suffices to show that, given one odd cycle, we can change the parity of any of the other $2g-1$ cycles in the basis while fixing the parity of the remaining $2g-2$. 

Suppose $\a_i$ is the cycle of odd degree given by assumption. We change only the parity of $b_i$ by Dehn twisting around $\a_i$, by Lemma \ref{lem:twist}. Now suppose we want to change the parity of only $b_{i+1}$. If $c_i$ is odd we Dehn twist around $\gamma_i$ to change the parity of $b_i, b_{i+1}$, and then twist again around $\a_i$ to change the parity of $b_i$ back. If $c_i$ is even then $a_i + a_{i+1}$ is even, so $a_{i+1}$ is odd, so we can twist around $\a_{i+1}$ to change the parity of $b_{i+1}$ only.

Once we can change the parity of only $b_{i+1}$ we can also only change the parity of only $a_{i+1}$. If $b_{i+1}$ is odd, twist around $\b_{i+1}$. If $b_{i+1}$ is even, change its parity, twist around $\b_{i+1}$ to change the parity of $a_{i+1}$, and change the parity of $b_{i+1}$ back to even.
  
Now suppose $k > i$ and we can change the parity of $b_k, a_k$ individually. Then we can assume $a_k$ is odd. If $c_k$ is odd we can change the parity of only $b_{k+1}$ by twisting around $\gamma_k$ and then changing the parity of $b_k$ back. If $c_k$ is even then $a_{k+1}$ is odd and we twist around $\a_{k+1}$ to change the parity of $b_{k+1}$. Similarly we can then change the parity of only $a_{k+1}$ at will by twisting around $\b_{k+1}$. For $k < i$ we make the same argument to show that we can change $a_k$, $b_k$ individually. 

Thus given any $a_i$ odd we can individually change the parity of all $a_j, b_j$ except $a_i$ itself. However if a specific $a_i$ is odd we can make any $a_j$ odd, and make the same argument. This gives us any configuration of parities where at least one $a_j$ is odd, $1 \le j \le g$. On the other hand if at least one $b_i$ is odd, similar arguments to the above allow us to achieve any configuration of parities except those where all $b_j$ are even, $1 \le j \le g$. The combination of these two cases proves the theorem. 
\end{proof}

\begin{corollary} \label{cor:22g-1}
Let $\Q_\l=\Q_g(2^m, k_1^{n_1}, k_2^{n_2},..., k_l^{n_l})$ where all the $k_i$ are even and $m \ge g$. Then $\Q_\l$ has exactly $2^{2g}-1$ connected components.
\end{corollary}

\begin{proof}
Theorem \ref{thm:c2} implies that the above strata have at most $2^{2g}-1$ connected components and Theorem \ref{thm:mcg} implies they have at least that many. 
\end{proof}

Recall that by the results of Section \ref{sec:up} we expected each of the components of Corollary \ref{cor:22g-1} to consist of quadratic differentials whose square roots are a section of a particular line bundle that squares to $K^2$. We may associate to each such bundle a $\Z_2$ valued linear functional on $H_1(M, \Z_2)$, in an analog of the spin structure construction of Section \ref{sec:back}, and this will coincide with the degree of the Gauss map. Thus it is not surprising that we should be able to use $Ga$ to classify connected components of the $\Q_\l$ with even zeroes.


\begin{thebibliography}{999}
\bibitem{EMZ} A. Eskin, H. Masur, A. Zorich. \emph{Moduli Spaces of Abelian Differentials: The Principal Boundary, Counting Problems and the Siegel-Veech Constants.} Publ. Math. Inst. Hautes Etudes Sci. No. 97 (2003), 61-179. 

\bibitem{GH} P. Griffiths, J. Harris. \emph{Principles of Algebraic Geometry.} Wiley \& Sons, 1994. 

\bibitem{KZ} M. Kontsevich, A. Zorich. \emph{Connected Components of
  the Moduli Spaces of Abelian Differentials with Prescribed
  Singularities.} Invent. Math. 153 (2003), no. 3, 631-678. 

\bibitem{L1} E. Lanneau. \emph{Hyperelliptic Components of the
  Moduli Spaces of Quadratic Differentials with Prescribed
  Singularities.} Comment. Math. Helv. 79 (2004), no. 3, 471--501. 

\bibitem{L2} E Lanneau. \emph{Parity of the Spin Structure Defined
  by a Quadratic Differential}. Geom. and Top., Vol. 8 (2004), 511-538.

\bibitem{MS} H. Masur, J. Smillie. Quadratic differentials with
  prescribed singularities and pseudo-Anosov diffeomorphisms, \emph{Comment. Math. Helv.} 68 (1993), no. 2, 289-307.

\bibitem{MT} H. Masur, S. Tabachnikov. Flat structures and rational
  billiards, in \emph{Handbook of Dynamical Systems}, Vol. 1A, 1015-1089.

\bibitem{MZ} H. Masur, A. Zorich. \emph{The principal boundary of the moduli
  spaces of quadratic differentials.} Publ. Math. Inst. Hautes
    Etudes Sci. 97 (2003), 61-179.

\bibitem{S} K. Strebel.  \emph{Quadratic Differentials.} Ergeb. Math. Grenzgeb. 3:5, Springer-Verlag, 1984.

\end{thebibliography}
\end{document}